\documentclass{amsart}
\usepackage{amssymb}
\usepackage{amsthm}
\newcommand{\pn}{\par\noindent}
\newcommand{\pmn}{\par\medskip\noindent}

\newtheorem{prop}{Proposition}[section]
\theoremstyle{definition} 
\newtheorem{defin}{Definition}[section]
 \theoremstyle{remark}
\newtheorem{rem}{Remark}[section]
\begin{document}
\title{Even and odd trees}
\author{Irina Busjatskaja \and Yury Kochetkov}
\begin{abstract} In this paper we at first consider plane trees
with the root vertex and a marked directed edge, outgoing from
the root vertex. For such trees we introduce a new characteristic
--- the \emph{parity}, using the bracket code. It turns out that
the parity depends only on the root vertex (not on the marked edge).
And in the case of an even number of vertices the parity does not
depend on the root vertex also.   Then we consider rotation
groups of bipartite trees, study their properties and prove that
in the case of even number of vertices rotation groups of even and
odd trees are different. \end{abstract}
\email{ibusjatskaja@hse.ru, yukochetkov@hse.ru} \maketitle
\section{Introduction}
\pn Let us consider a tree with $n$ vertices embedded into the
plane. One of vertices will be a root vertex. Such tree will be
represented on the plane as "growing" downwards from the root
vertex. The leftmost edge outgoing from the root vertex will be
called the marked edge. Such trees will be called pr-trees (i.e.
plane rooted trees). The root vertex will be the vertex of level
0, vertices adjacent to the root will be vertices of level 1, and
so on.  \pmn Vertices of level 1 are root vertices of pr-trees
(subtrees) of the original pr-tree. By a \emph{level one
permutation} we will understand any change of the order of level 1
vertices (and corresponding pr-trees). \pmn We can describe
pr-trees using proper bracket structures (bracket codes) with the
number of opening brackets (and the number of closing brackets)
equal to the number of vertices. A bracket structure is proper, if
the number of opening brackets is not less, than the number of
closing ones in any left segment of this structure. The bracket
structure is constructed in the following way: beginning at root
vertex, we make the counterclockwise walkabout (or simply the
walk) around the tree, i.e. the tree is always on our left. The
length of this walkabout (measured by edges) is $2(n-1)$, where
$n$ is the number of vertices: we pass each edge twice in two
opposite directions. When we meet a vertex for the first time we
open bracket. When we meet a vertex for the last time we close the
bracket. The first opening bracket and the last closing bracket
correspond to the root vertex.  Here is an example of this
correspondence.
\[\begin{picture}(110,40) \put(0,25){\circle*{2}}
\put(10,15){\circle*{2}} \multiput(10,35)(10,-10){4}{\circle*{2}}
\put(0,25){\line(1,1){10}} \put(10,15){\line(1,1){10}}
\put(10,35){\line(1,-1){30}} \put(60,20){$\Rightarrow$}
\put(90,20){\small (()(()(())))} \end{picture}\]

\begin{center}{\large Figure 1}\end{center}

\begin{defin} A pair "closing bracket before opening bracket (not
necessary adjacent)" in a bracket structure is called an
\emph{inversion}.
\end{defin}

\begin{rem} There are 6 inversions in the bracket code of the tree in
the Figure 1. \end{rem}

\begin{defin} A pr-tree is called \emph{even}, if the number of
inversions in the corresponding bracket structure is even, and is
called \emph{odd} in the opposite case. \end{defin}

\begin{rem} The tree in the Figure 1 is even. \end{rem}
\pmn The following properties of the parity will be proved.

\begin{itemize} \item A level one permutation does not change the
parity of a pr-tree (Proposition 2.1). \item If the number of
vertices $n$ is even, then the parity does not depend on the
choice of the root vertex (Proposition 2.2).\end{itemize}

\begin{rem} From propositions 2.1 and 2.2 it follows that if two pr-trees
with even number of vertices are graph-isomorphic, then they have
the same parity. \end{rem}

\pmn Let $A$, $B$ and $C$ be three vertices of a pr-tree (see
Figure 2), such that: a) the degree of $B$ is one; b) during the
walk we meet vertices in the sequence $A,B,A,C$ as in the figure
below.
\[\begin{picture}(140,70) \put(10,45){\circle*{2}}
\put(30,45){\circle*{2}} \put(30,25){\circle*{2}}
\put(10,45){\line(1,0){20}} \put(30,15){\line(0,1){40}}
\put(29,0){$\vdots$} \put(29,60){$\vdots$} \put(2,42){\scriptsize
B} \put(35,42){\scriptsize A} \put(35,22){\scriptsize C}

\put(60,32){or}

\put(100,25){\circle*{2}} \put(100,45){\circle*{2}}
\put(100,15){\line(0,1){40}} \put(99,0){$\vdots$}
\put(99,60){$\vdots$} \put(90,22){\scriptsize A}
\put(90,42){\scriptsize C} \put(130,25){\circle*{2}}
\put(100,25){\line(1,0){30}} \put(134,22){\scriptsize B}
\end{picture}\]
\begin{center}{\large Figure 2}\end{center}
\pmn
\begin{defin} A \emph{transposition} is a translation of the vertex
$B$ with the attached edge along one(!) edge of the tree (as in
the walkabout) from vertex $A$ to the next vertex $C$ (as in the
Figure 3).
\[\begin{picture}(180,70) \put(10,45){\circle*{2}}
\put(30,45){\circle*{2}} \put(30,25){\circle*{2}}
\put(10,45){\line(1,0){20}} \put(30,15){\line(0,1){40}}
\put(29,0){$\vdots$} \put(29,60){$\vdots$} \put(2,42){\scriptsize
B} \put(35,42){\scriptsize A} \put(35,22){\scriptsize C}
\put(30,45){\line(-1,1){10}} \put(30,25){\line(-1,-1){10}}
\multiput(18,57)(-3,3){3}{\circle*{1}}
\multiput(18,13)(-3,-3){3}{\circle*{1}}

\put(60,30){$\xrightarrow{\text{\small transposition}}$}

\put(142,22){\scriptsize B} \put(150,25){\circle*{2}}
\put(170,45){\circle*{2}} \put(170,25){\circle*{2}}
\put(150,25){\line(1,0){20}} \put(170,15){\line(0,1){40}}
\put(169,0){$\vdots$} \put(169,60){$\vdots$}
\put(175,42){\scriptsize A} \put(175,22){\scriptsize C}
\put(170,45){\line(-1,1){10}} \put(170,25){\line(-1,-1){10}}
\multiput(158,57)(-3,3){3}{\circle*{1}}
\multiput(158,13)(-3,-3){3}{\circle*{1}} \end{picture}\] \pmn
\begin{center}{\large Figure 3}\end{center}
\pmn The transformation of the right tree in Figure 2 is
analogous.
\end{defin}
\pn It will be proved (Proposition 3.1) that a transposition
changes the parity.
\begin{rem} The standard chain-tree is obviously even (see Figure
4).
\[\begin{picture}(10,50) \put(5,5){\circle*{2}}
\put(5,15){\circle*{2}} \put(5,5){\line(0,1){15}}
\multiput(5,23)(0,3){3}{\circle*{1}} \put(5,32){\line(0,1){15}}
\put(5,37){\circle*{2}} \put(5,47){\circle*{2}}
\end{picture}\]
\begin{center}{\large Figure 4}\end{center}
\pmn Thus, we can determine the parity of any pr-tree by counting
the number of transpositions that transform our tree to the
chain-tree.\end{rem}
\begin{defin} A plane tree $T$ is a \emph{sum} of two plane trees
$T_1$ and $T_2$, if after deletion of some edge of $T$ we obtain
the disjoint union of $T_1$ and $T_2$.
\end{defin} \pn It will be proved (Proposition 4.1) that the sum
of two trees with even number of vertices each is odd only when
one of the summands is odd and another is even.
\begin{rem} By representing a tree as a sum, we can simplify the
determination of its parity.\end{rem} \pn Let us consider a tree
$T$ with $m$ edges. The \emph{passport} (see \cite{LZ}) of $T$ is
the list $[a_1,\ldots,a_m]$, where $a_i$, $i=1,\ldots,m$, is the
number of vertices of degree $i$. \pmn A bipartite tree is a tree
with vertices colored in two colors --- black and white so, that
adjacent vertices have different colors. The passport of a
bipartite tree with $m$ edges is the list
$[[a_1,\ldots,a_m],[b_1,\ldots,b_m]]$, where $a_i$,
$i=1,\ldots,m$, is the number of white vertices of degree $i$ and
$b_i$, $i=1,\ldots,m$, is the number of black vertices of degree
$i$.
\begin{defin} Let $T$ be a plane bipartite tree with $m$ edges, enumerated
in some arbitrary way. Let $M=\{1,2,\ldots,m\}$. The permutation
$s_w:M\to M$ is defined in the following way: $s_w(i)=j$, if
$i$-th and $j$-th edges are incident to the same white vertex $v$
and $j$-th edge is the next after $i$-th edge in the
counterclockwise going around $v$. The permutation $s_b$ is
defined in the same way, but with respect to black vertices.  The
rotation group (see \cite{LZ}) $R(T)$ of $T$ is the subgroup of
$S_m$, generated by permutations $s_w$ and $s_b$, i.e.
$R(T)=\langle s_w,s_b\rangle\subset S_m$.
\end{defin}
\begin{rem} The product $s_ws_b$ is an $m$-cycle.
Different numerations of edges give us conjugate rotation groups.
\end{rem}
\pn It will be proved (Propositions 5.1) that rotations groups of
two plane trees with the same even number of vertices, but of
different parities, are different.
\begin{defin} A bipartite tree is called \emph{clean}, if all its
black vertices have degree two (see \cite{LZ}). \end{defin} \pn
For any given tree $T$, we can construct a bipartite clean tree
$T^c$ by dividing each edge by black vertex into two edges
(initial vertices now are white). Let $T$ and $U$ be two plane
trees with $n$ vertices and $n$ is even. If $T$ and $U$ are of
different parities, then rotation groups of $T^c$ and $U^c$ are
different (Proposition 6.1). \pmn In Supplement 1 we demonstrate
that trees with different parities may be cospectral. \pmn In
Supplement 2 we introduce graph $G_n$ (for even n) --- the "graph
of trees" and study its planarity for small $n$.

\section{Even and odd rooted trees}
\pn
\begin{prop} A level one permutation does not change the parity of a
pr-tree $T$. \end{prop}
\begin{proof} Let $v_0$ be the root vertex of a pr-tree $T$,
$v_1,\ldots,v_s$ be level one vertices (enumerated from left to
the right) and $T_1,\ldots,T_s$ be the corresponding pr-subtrees
(Figure 5).
\[\begin{picture}(150,70) \put(0,5){\line(1,0){30}}
\put(50,5){\line(1,0){30}} \put(120,5){\line(1,0){30}}
\put(0,22){\line(1,0){30}} \put(50,22){\line(1,0){30}}
\put(120,22){\line(1,0){30}} \put(0,5){\line(0,1){17}}
\put(30,5){\line(0,1){17}} \put(50,5){\line(0,1){17}}
\put(80,5){\line(0,1){17}} \put(120,5){\line(0,1){17}}
\put(150,5){\line(0,1){17}} \put(12,10){\small $T_1$}
\put(62,10){\small $T_2$} \put(132,10){\small $T_s$}
\put(75,55){\circle*{2}} \put(72,60){\small $v_0$}
\multiput(88,13)(8,0){4}{\circle*{2}} \put(15,25){\line(2,1){60}}
\put(65,25){\line(1,3){10}} \put(135,25){\line(-2,1){60}}
\end{picture}\]
\begin{center}{\large Figure 5}\end{center}
\pmn Let a subtree $T_i$ has $k_i$ vertices and its bracket
structure has $x_i$ inversions. Then the bracket structure of the
tree $T$ has
$$\sum_i x_i+\sum_{i<j}k_ik_j$$ inversions. If we interchange
subtrees $T_i$ and $T_{i+1}$, then the first sum will be the same
and in the second sum the term $k_ik_{i+1}$ will be replaced by
$k_{i+1}k_i$. Hence, the parity of a pr-tree does not depend on
the choice of marked edge. \end{proof}

\begin{prop} If a pr-tree has an even number of vertices $n$ then its
parity does not depend on the choice of root vertex. \end{prop}
\begin{proof} Let $O$ be the root vertex and $OP$ be the marked edge.
We make $P$ the new root vertex (and preserve the order of the walk).
Then the tree will be transformed in the following way:
\[\begin{picture}(320,80) \qbezier[35](0,5)(20,5)(40,5)
\qbezier[35](0,30)(20,30)(40,30) \qbezier[20](0,5)(0,20)(0,30)
\qbezier[20](40,5)(40,20)(40,30) \put(18,15){A}

\put(20,55){\circle*{3}} \put(20,55){\line(-1,-2){10}}
\put(20,55){\line(0,-1){20}} \put(20,55){\line(1,-2){10}}
\put(40,65){\circle*{3}} \put(37,70){\small O} \put(15,58){\small
P} \put(40,65){\line(-2,-1){20}}

\qbezier[50](60,5)(90,5)(120,5) \qbezier[50](60,30)(90,30)(120,30)
\qbezier[20](60,5)(60,20)(60,30)
\qbezier[20](120,5)(120,20)(120,30) \put(88,15){B}

\put(40,65){\line(1,-1){30}} \put(40,65){\line(5,-3){50}}
\qbezier(40,65)(75,50)(110,35)

\put(150,30){$\Rightarrow$}

\qbezier[40](190,5)(215,5)(240,5)
\qbezier[40](190,30)(215,30)(240,30)
\qbezier[20](190,5)(190,20)(190,30)
\qbezier[20](240,5)(240,20)(240,30) \put(213,15){A}

\put(250,65){\circle*{3}} \put(247,70){\small P}
\put(250,65){\line(-5,-3){50}} \put(250,65){\line(-1,-1){30}}
\put(250,65){\line(-2,-3){20}} \put(250,65){\line(4,-1){40}}
\put(290,55){\circle*{3}} \put(293,57){\small O}

\qbezier[45](260,5)(290,5)(320,5)
\qbezier[45](260,30)(290,30)(320,30)
\qbezier[20](260,5)(260,20)(260,30)
\qbezier[20](320,5)(320,20)(320,30) \put(288,15){B}

\put(290,55){\line(-1,-1){20}} \put(290,55){\line(0,-1){20}}
\put(290,55){\line(1,-1){20}}
\end{picture}\]
\begin{center}{\large Figure 6}\end{center}\pn
The corresponding change of bracket code is as follows:
$$\begin{array}{cccccc} (&(&\cdots&)&\cdots&)\\
&\uparrow&\text{\scriptsize A}&\uparrow&\text{\scriptsize B}&\\
&\text{\scriptsize P}&&\text{\scriptsize P}&&\end{array} \qquad
\Rightarrow
\qquad \begin{array}{cccccc} (&\cdots&(&\cdots&)&)\\
&\text{\scriptsize A}&\uparrow&\text{\scriptsize B}&\uparrow&\\
&&\text{\scriptsize O}&&\text{\scriptsize O}&\end{array}$$ If
there are $k$ vertices in the forest $A$ and $s$ inversions in its
bracket code, and if there are $l$ vertices in the forest $B$ and
$t$ inversions in its bracket code, then the left and the right
trees in Figure 6 have $s+kl+l+t$ and $s+k+kl+t$ inversions in
their bracket codes, respectively. As $n=k+l+2$, then numbers
$s+kl+l+t$ and $s+k+kl+t$ have the same parity. \end{proof}

\section{Transpositions}

\begin{prop} A transposition changes the parity of a pr-tree.
\end{prop}
\begin{proof} We can see in the Figure 2 that there are two possibilities:
either we meet $C$ for the first time in the walk around the tree,
or we meet $A$ for the last time. The change of the bracket code
(see Figure 3) in the first case is as follows:
$$\begin{array}{ccccc} \cdots&(&)&(&\cdots\\ &\uparrow&\uparrow&\uparrow&\\
&\text{\scriptsize B}&\text{\scriptsize B}&\text{\scriptsize C}&
\end{array} \quad \Rightarrow\quad
\begin{array}{ccccc} \cdots&(&(&)&\cdots\\ &\uparrow&\uparrow&\uparrow&\\
&\text{\scriptsize C}&\text{\scriptsize B}&\text{\scriptsize B}&
\end{array}$$ and in the second ---
$$\begin{array}{ccccc} \cdots&(&)&)&\cdots\\ &\uparrow&\uparrow&\uparrow&\\
&\text{\scriptsize B}&\text{\scriptsize B}&\text{\scriptsize A}&
\end{array} \quad \Rightarrow\quad
\begin{array}{ccccc} \cdots&)&(&)&\cdots\\ &\uparrow&\uparrow&\uparrow&\\
&\text{\scriptsize A}&\text{\scriptsize B}&\text{\scriptsize B}&
\end{array}$$ In both cases the number of inversions changes by 1.
\end{proof}

\section{Sums}
\pn Here we will study the sum of two trees $T_1$ and $T_2$ (see
Definition 1.4).
\begin{prop} If a tree $T_1$ has an even number of vertices $k$ and a
tree $T_2$ also has an even number of vertices $l$, then their sum
is odd only when one of summands is odd and another --- even.
\end{prop}
\begin{proof} Let $T$ be the sum of $T_1$ and $T_2$ and let $e$ be the
edge, connecting $T_1$ and $T_2$. $T$ has $k+l$ vertices --- an
even number, hence, by Proposition 2.2 we can choose any vertex as
the root vertex. So the root vertex $O$ will be the vertex of
$T_1$, incident $e$. By Proposition 2.1 any edge, incident to $O$,
can be chosen as a marked edge, so the marked edge will be the
next after $e$ in going counterclockwise around $O$ (i.e. in the
walk around $T$ we at first make the walk around $T_1$). \pmn The
number of inversions is $s+t+kl$, where $s$ and $t$ are number of
inversions in bracket codes of $T_1$ and $T_2$, respectively. It
remains to note that $kl$ is even.
\end{proof}

\section{Rotation group}
\pn Here we will consider trees with an even number of vertices
$n$. $R(T)$ will be the rotation group of a tree $T$ (see
Definition 1.5).
\begin{prop} Even and odd trees with $n$ vertices have different
rotation groups. \end{prop}
\begin{proof} If $T$ is a tree with $n$ vertices, then it has an
odd number of edges $m=n-1$. Let us consider any bipartite
structure on $T$ (there are two of them). The permutation
$t=s_ws_b$ is an $m$-cycle, hence it is an even permutation, thus,
permutations $s_w$ and $s_b$ are both odd or both even. Let $U$ be
a tree, obtained from $T$ by one transposition: a vertex of degree
1 with attached edge moves from white vertex to the nearest black
one. The trees $T$ and $U$ have different parities by Proposition
3.1. After the transposition the number of even-length cycles in
the cyclic presentation of $s_w$ changes by one and the number of
even-length cycles in the cyclic presentation of $s_b$ is also
changes by one. Thus, rotations groups of trees of one parity
belong to the alternating group $A_m$ and rotation groups of trees
of another parity do not belong to $A_m$. \end{proof} \pmn {\bf
Conclusion} \emph{If $n\equiv 2\text{ mod }4$, then rotation
groups of even trees with $n$ vertices belong the alternating
group. If $n\equiv 0\text{ mod }4$, then rotation groups of odd
trees belong to the alternating group.}
\begin{proof} Let consider the chain-tree (see Figure 4) with $n$
vertices. This tree is even. It has one white vertex of degree
one, one black vertex of degree one, $(n-2)/2$ white vertices of
degree 2 and $(n-2)/2$ black vertices of degree 2. Permutations
$s_w$ and $s_b$ are even, if the number $(n-2)/2$ is even. It
remains to note that By Remark 1.4 any tree can be transformed in
the chain tree by a sequence of transpositions.
\end{proof}
\begin{rem} Let $T$ and $U$ be trees with $n$ vertices, one even and one
odd. Then their passports as bipartite trees are different.
Indeed, one tree has even number of white vertices of even degrees
and even number of black vertices of even degrees and another has
odd number of white vertices of even degrees and odd number of
black vertices of even degrees. \end{rem}

\section{Clean trees}
\begin{prop} If two trees $T$ and $U$ with an even number $n$ of
vertices have different parities, then rotation groups of trees
$T^c$ and $U^c$ are different.\end{prop}
\begin{proof} Let us consider some enumeration of $T^c$ edges.
As the permutation $t=s_ws_b$ is a $(2n-2)$-cycle, then it is odd.
As $s_b$ is a product of $n-1$ --- and odd number, of $2$-cycles,
then it is odd also, thus $s_w$ is even. \pmn The tree $T$ has its
own bipartite structure with black and white vertices (in $T^c$
all these vertices are white). We will call them pw-vertices
(previously white) and pb-vertices (previously black). Each edge
of $T$ is divided by black vertex in two edges of $T^c$. Thus,
half of $T^c$ edges are incident to pw-vertices and half --- to
pb-vertices. Let $N_w$ be the set of $T^c$ edges, incident to
pw-vertices, and $N_b$ be the set of $T^c$ edges, incident to
pb-vertices. As the permutation $s_w$ permutes separately elements
of $N_w$ and $N_b$, then it generates two permutations: the
permutation $\sigma_w$ of the set $N_w$ and the permutation
$\sigma_b$ of the set $N_b$. As $s_w=\sigma_w\sigma_b$, then
permutations $\sigma_w$ and $\sigma_b$ are both even or both odd.
A transposition on $T$ changes the degree of a pw-vertex by one and
the degree of a pb-vertex by one also, thus, it simultaneously
changes parities of $\sigma_w$ and $\sigma_b$. \pmn
Let $A_w$ ($A_b$) be the alternating group of permutations of the
set $N_w$ ($N_b$) and let $G=A_wA_b$. Then $s_w$ belongs to $G$
for trees of one parity and does not belong to $G$ for trees of another.
As $e\in N_w$ only when $t(e)\in N_b$, then $t^2\in G$. As
$R(T^c)=\langle t,s_w\rangle$, then $R(T^c)\subset tG\cup G$ for
trees of one parity and $R(T^c)\not\subset tG\cup G$ for trees of
another. It remains to note that this result does not depend on
the enumeration of $T^c$ edges.  \end{proof}

\begin{center}{\Large Supplement 1. Characteristic
polynomial}\end{center} \pmn One can ask a natural question: can
the spectrum of on odd tree be the same as the spectrum of an even
tree? Two non-isomorphic trees with the same characteristic
polynomial are called \emph{cospectral} (see \cite{DH}). \pmn We
can easily find two cospectral trees with characteristic
polynomial $x^8-7x^6+9x^4$
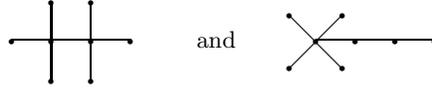
\[\begin{picture}(160,40) \multiput(0,20)(15,0){4}{\circle*{2}}
\put(15,5){\circle*{2}} \put(15,35){\circle*{2}}
\put(30,5){\circle*{2}} \put(30,35){\circle*{2}}
\put(0,21){\line(1,0){45}} \put(15,5){\line(0,1){30}}
\put(30,5){\line(0,1){30}} \put(70,18){\small and}
\multiput(115,20)(15,0){4}{\circle*{2}} \put(105,10){\circle*{2}}
\put(105,30){\circle*{2}} \put(125,10){\circle*{2}}
\put(125,30){\circle*{2}} \put(115,21){\line(1,0){45}}
\put(105,10){\line(1,1){20}} \put(105,30){\line(1,-1){20}}
\end{picture}\] but both of them are even.
\pmn However, there exit two cospectral trees with twelve vertices
with characteristic polynomial
$x^{12}-11x^{10}+42x^8-66x^6+39x^4-6x^2$.
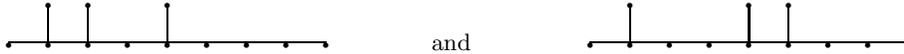
\[\begin{picture}(340,25) \multiput(0,5)(15,0){9}{\circle*{2}}
\multiput(220,5)(15,0){9}{\circle*{2}} \put(0,6){\line(1,0){120}}
\put(220,6){\line(1,0){120}} \put(160,3){\small and}
\put(15,20){\circle*{2}} \put(30,20){\circle*{2}}
\put(60,20){\circle*{2}} \put(235,20){\circle*{2}}
\put(280,20){\circle*{2}} \put(295,20){\circle*{2}}
\put(15,5){\line(0,1){15}} \put(30,5){\line(0,1){15}}
\put(60,5){\line(0,1){15}} \put(235,5){\line(0,1){15}}
\put(280,5){\line(0,1){15}} \put(295,5){\line(0,1){15}}
\end{picture}\] The tree on the left is even and the tree on the
right is odd.

\begin{center}{\Large Supplement 2. Graph of trees}\end{center}
\pmn For an even $n$ we can define the graph $G_n$. The set of
vertices of $G_n$ is the set of all pairwise non-isomorphic trees
with $n$ vertices. Two vertices of $G_n$ are adjacent if a
transposition transforms one of the corresponding trees into
another. $G_n$ will be called the \emph{graph of trees}. By
Proposition 3.1 $G_n$ is bipartite. \pmn For $n=6,8$ we can
construct this graphs. The graph $G_6$ is quite simple.
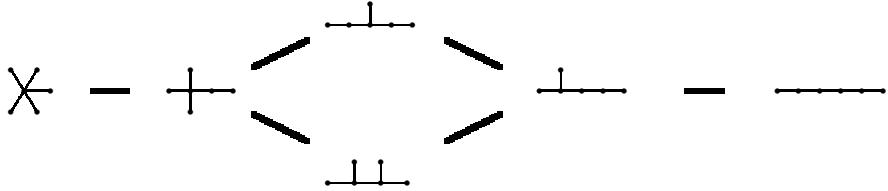
\[\begin{picture}(330,80) \put(0,48){\circle*{2}}
\put(0,32){\circle*{2}} \put(5,40){\circle*{2}}
\put(10,48){\circle*{2}} \put(10,32){\circle*{2}}
\put(15,40){\circle*{2}} \qbezier(0,32)(5,40)(10,48)
\qbezier(0,48)(5,40)(10,32) \put(5,40){\line(1,0){10}}

\multiput(60,40)(8,0){4}{\circle*{2}} \put(68,48){\circle*{2}}
\put(68,32){\circle*{2}} \put(60,40){\line(1,0){24}}
\put(68,32){\line(0,1){16}}

\multiput(120,65)(8,0){5}{\circle*{2}} \put(136,73){\circle*{2}}
\put(120,65){\line(1,0){32}} \put(136,65){\line(0,1){8}}

\multiput(120,5)(10,0){4}{\circle*{2}} \put(130,13){\circle*{2}}
\put(140,13){\circle*{2}} \put(120,5){\line(1,0){30}}
\put(130,5){\line(0,1){8}} \put(140,5){\line(0,1){8}}

\multiput(200,40)(8,0){5}{\circle*{2}} \put(208,48){\circle*{2}}
\put(200,40){\line(1,0){32}} \put(208,40){\line(0,1){8}}

\multiput(290,40)(8,0){6}{\circle*{2}}
\put(290,40){\line(1,0){40}}

\linethickness{0.7mm} \put(30,40){\line(1,0){15}}
\qbezier(92,49)(102,54)(112,59) \qbezier(92,31)(102,26)(112,21)
\qbezier(165,21)(175,26)(185,31) \qbezier(165,59)(175,54)(185,49)
\put(255,40){\line(1,0){15}}
\end{picture}\] The graph $G_8$ has 23 vertices and 37 edges.
All pairwise non-isomorphic trees are presented below.
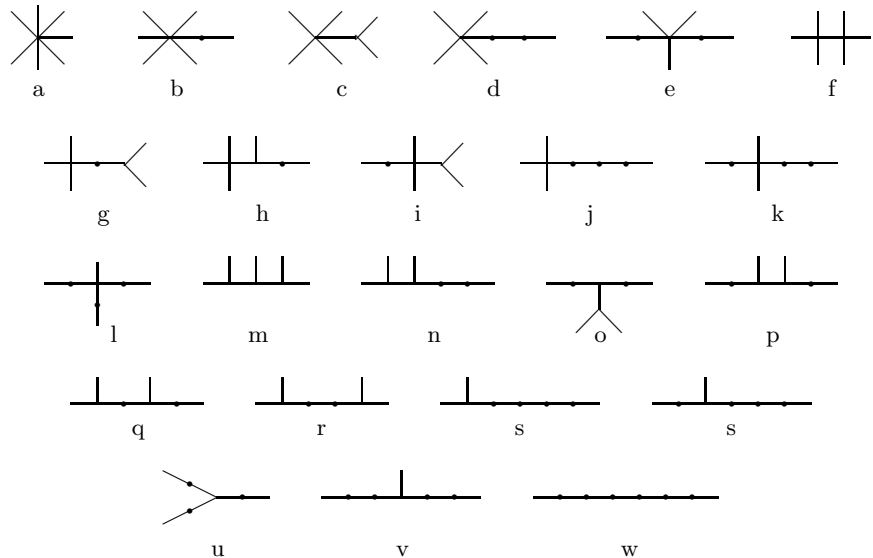
\[\begin{picture}(325,40) \put(0,15){\line(1,1){20}}
\put(0,35){\line(1,-1){20}} \put(10,13){\line(0,1){24}}
\put(10,25){\line(1,0){13}} \put(8,3){\small a}

\put(50,15){\line(1,1){20}} \put(50,35){\line(1,-1){20}}
\put(48,25){\line(1,0){36}} \put(72,25){\circle*{2}}
\put(60,3){\small b}

\put(105,15){\line(1,1){20}} \put(105,35){\line(1,-1){20}}
\put(115,25){\line(1,0){15}} \put(130,26){$\diagup$}
\put(130,19){$\diagdown$} \put(123,3){\small c}

\put(160,15){\line(1,1){20}} \put(160,35){\line(1,-1){20}}
\put(170,25){\line(1,0){36}} \put(182,25){\circle*{2}}
\put(194,25){\circle*{2}} \put(180,3){\small d}

\put(225,25){\line(1,0){48}} \put(249,13){\line(0,1){12}}
\put(249,25){\line(1,1){10}} \put(249,25){\line(-1,1){10}}
\put(237,25){\circle*{2}} \put(261,25){\circle*{2}}
\put(247,3){\small e}

\put(295,25){\line(1,0){30}} \put(305,15){\line(0,1){20}}
\put(315,15){\line(0,1){20}} \put(309,3){\small f}
\end{picture}\]
\[\begin{picture}(300,40) \put(0,25){\line(1,0){30}}
\put(10,15){\line(0,1){20}}  \put(30,26){$\diagup$}
\put(30,18){$\diagdown$} \put(20,25){\circle*{2}}
\put(20,3){\small g}

\put(60,25){\line(1,0){40}} \put(70,15){\line(0,1){20}}
\put(80,25){\line(0,1){10}} \put(90,25){\circle*{2}}
\put(80,3){\small h}

\put(120,25){\line(1,0){30}} \put(140,15){\line(0,1){20}}
\put(130,25){\circle*{2}} \put(150,18){$\diagdown$}
\put(150,26){$\diagup$} \put(140,3){\small i}

\put(180,25){\line(1,0){50}} \put(190,15){\line(0,1){20}}
\multiput(200,25)(10,0){3}{\circle*{2}} \put(205,3){\small j}

\put(250,25){\line(1,0){50}} \put(270,15){\line(0,1){20}}
\put(260,25){\circle*{2}} \put(280,25){\circle*{2}}
\put(290,25){\circle*{2}} \put(275,3){\small k}
\end{picture}\]
\[\begin{picture}(300,40) \put(0,25){\line(1,0){40}}
\put(20,9){\line(0,1){24}} \put(10,25){\circle*{2}}
\put(30,25){\circle*{2}} \put(20,17){\circle*{2}}
\put(25,3){\small l}

\put(60,25){\line(1,0){40}}
\multiput(70,25)(10,0){3}{\line(0,1){10}} \put(77,3){\small m}

\put(120,25){\line(1,0){50}}
\multiput(130,25)(10,0){2}{\line(0,1){10}}
\put(150,25){\circle*{2}} \put(160,25){\circle*{2}}
\put(145,3){\small n}

\put(190,25){\line(1,0){40}} \put(200,25){\circle*{2}}
\put(220,25){\circle*{2}} \put(210,25){\line(0,-1){10}}
\put(210,8){$\diagdown$} \put(201,8){$\diagup$} \put(208,3){\small
o}

\put(250,25){\line(1,0){50}} \put(260,25){\circle*{2}}
\put(290,25){\circle*{2}}  \put(270,25){\line(0,1){10}}
\put(280,25){\line(0,1){10}} \put(273,3){\small p}
\end{picture}\]
\[\begin{picture}(280,30) \put(0,15){\line(1,0){50}}
\put(10,15){\line(0,1){10}} \put(30,15){\line(0,1){10}}
\put(20,15){\circle*{2}} \put(40,15){\circle*{2}}
\put(23,3){\small q}

\put(70,15){\line(1,0){50}} \put(80,15){\line(0,1){10}}
\put(110,15){\line(0,1){10}} \put(90,15){\circle*{2}}
\put(100,15){\circle*{2}} \put(93,3){\small r}

\put(140,15){\line(1,0){60}} \put(150,15){\line(0,1){10}}
\multiput(160,15)(10,0){4}{\circle*{2}} \put(168,3){\small s}

\put(220,15){\line(1,0){60}} \put(240,15){\line(0,1){10}}
\multiput(250,15)(10,0){3}{\circle*{2}} \put(230,15){\circle*{2}}
\put(248,3){\small s}
\end{picture}\]
\[\begin{picture}(210,40) \put(0,35){\line(2,-1){20}}
\put(0,15){\line(2,1){20}} \put(20,25){\line(1,0){20}}
\put(10,20){\circle*{2}} \put(10,30){\circle*{2}}
\put(30,25){\circle*{2}} \put(18,3){\small u}

\put(60,25){\line(1,0){60}} \put(90,25){\line(0,1){10}}
\multiput(70,25)(10,0){2}{\circle*{2}}
\multiput(100,25)(10,0){2}{\circle*{2}} \put(88,3){\small v}

\put(140,25){\line(1,0){70}}
\multiput(150,25)(10,0){6}{\circle*{2}}  \put(173,3){\small w}
\end{picture}\] It turns out that the graph $G_8$ is planar:
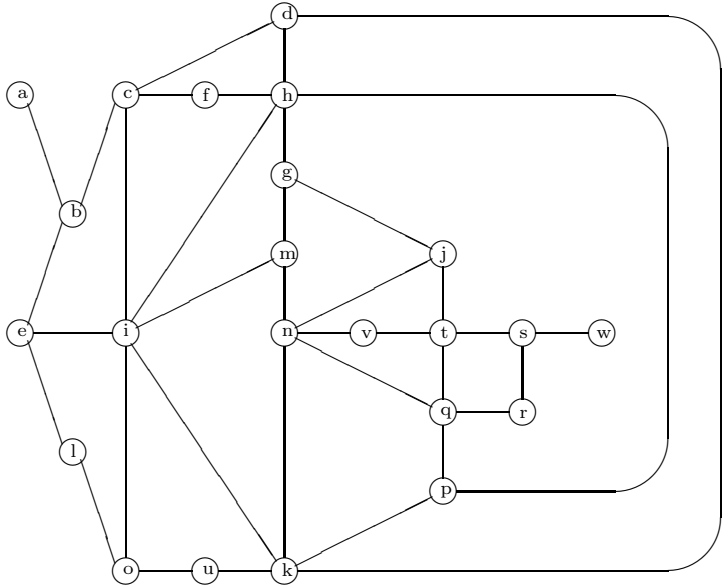
\[\begin{picture}(270,230) \put(100,10){\circle{10}}
\multiput(100,100)(0,30){5}{\circle{10}} \put(99,8){\scriptsize k}
\put(99,99){\scriptsize n} \put(98,128){\scriptsize m}
\put(99,159){\scriptsize g} \put(99,188){\scriptsize h}
\put(99,219){\scriptsize d}

\multiput(160,40)(0,30){4}{\circle{10}} \put(159,39){\scriptsize
p} \put(159,69){\scriptsize q} \put(159,98){\scriptsize t}
\put(159,128){\scriptsize j} \put(130,100){\circle{10}}
\put(129,98){\scriptsize v}

\multiput(100,105)(0,30){4}{\line(0,1){20}}
\put(100,15){\line(0,1){80}}
\multiput(160,45)(0,30){3}{\line(0,1){20}}
\multiput(105,100)(30,0){2}{\line(1,0){20}}
\put(104,12){\line(2,1){52}} \put(104,102){\line(2,1){52}}
\put(104,98){\line(2,-1){52}} \put(104,158){\line(2,-1){52}}

\multiput(190,70)(0,30){2}{\circle{10}} \put(220,100){\circle{10}}
\put(189,68){\scriptsize r} \put(189,98){\scriptsize s}
\put(218,98){\scriptsize w}
\multiput(165,70)(0,30){2}{\line(1,0){20}}
\put(190,75){\line(0,1){20}} \put(195,100){\line(1,0){20}}

\put(105,190){\line(1,0){60}} \put(165,115){\oval(160,150)[r]}
\put(105,115){\oval(320,210)[r]}

\put(40,100){\circle{10}} \put(39,99){\scriptsize i}
\put(44,102){\line(2,1){52}} \put(42,104){\line(2,3){55}}
\put(42,96){\line(2,-3){55}}

\multiput(40,10)(30,0){2}{\circle{10}}
\multiput(40,190)(30,0){2}{\circle{10}}
\multiput(45,10)(30,0){2}{\line(1,0){20}}
\multiput(45,190)(30,0){2}{\line(1,0){20}} \put(39,9){\scriptsize
o} \put(69,9){\scriptsize u} \put(39,189){\scriptsize c}
\put(69,188){\scriptsize f}

\put(44,192){\line(2,1){52}} \put(40,15){\line(0,1){80}}
\put(40,105){\line(0,1){80}} \put(0,100){\circle{10}}
\put(-1,99){\scriptsize e} \put(5,100){\line(1,0){30}}
\put(20,145){\circle{10}} \put(19,144){\scriptsize b}
\put(3,103){\line(1,3){13}} \put(23,148){\line(1,3){13}}
\put(20,55){\circle{10}} \put(19,53){\scriptsize l}
\put(3,97){\line(1,-3){13}} \put(23,52){\line(1,-3){13}}

\put(0,190){\circle{10}} \put(-1,189){\scriptsize a}
\put(3,187){\line(1,-3){13}}
\end{picture}\] However, the graph $G_{10}$ with 106 vertices and 238
edges is not planar. In order to demonstrate its non-planarity we
will construct a subgraph, homeomorphic to the graph $K_{3,3}$
(see \cite{Ha}), i.e. to the bipartite graph
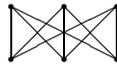
\[\begin{picture}(40,30) \multiput(0,5)(20,0){3}{\circle*{2}}
\multiput(0,25)(20,0){3}{\circle*{2}}
\multiput(0,5)(20,0){3}{\line(0,1){20}}
\multiput(0,5)(20,0){2}{\line(1,1){20}}
\multiput(0,25)(20,0){2}{\line(1,-1){20}}
\put(0,5){\line(2,1){40}} \put(0,25){\line(2,-1){40}}
\end{picture}\]\pn 1. Let us consider trees $a$, $b$ and $c$ up to
isomorphism
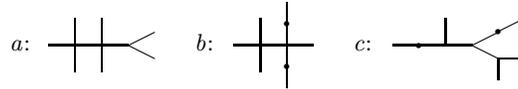
\[\begin{picture}(195,40) \put(0,18){\small\emph a:}
\put(15,20){\line(1,0){30}} \put(25,10){\line(0,1){20}}
\put(35,10){\line(0,1){20}} \put(45,20){\line(2,1){10}}
\put(45,20){\line(2,-1){10}}

\put(70,18){\small\emph b:} \put(85,20){\line(1,0){30}}
\put(95,10){\line(0,1){20}} \put(105,4){\line(0,1){32}}
\put(105,12){\circle*{2}} \put(105,28){\circle*{2}}

\put(130,18){\small\emph c:} \put(145,20){\line(1,0){30}}
\put(155,20){\circle*{2}} \put(165,20){\line(0,1){10}}
\put(175,20){\line(2,1){20}} \put(185,25){\circle*{2}}
\put(175,20){\line(2,-1){10}} \put(185,15){\line(1,0){8}}
\put(185,15){\line(0,-1){8}} \end{picture}\] Each of these trees
can be transformed by one transposition into the tree $d$
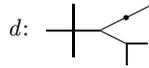
\[\begin{picture}(55,40) \put(15,20){\line(1,0){20}}
\put(25,10){\line(0,1){20}} \put(35,20){\line(2,1){20}}
\put(45,25){\circle*{2}} \put(35,20){\line(2,-1){10}}
\put(45,15){\line(1,0){8}} \put(45,15){\line(0,-1){8}}
\put(0,18){\small\emph d:}
\end{picture}\] Trees $a$ and $b$ can be transformed by one
transposition into the tree $e$
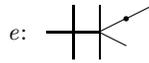
\[\begin{picture}(55,30) \put(15,15){\line(1,0){20}}
\put(25,5){\line(0,1){20}} \put(35,15){\line(2,1){20}}
\put(45,20){\circle*{2}} \put(35,15){\line(2,-1){10}}
\put(35,5){\line(0,1){20}} \put(0,12){\small\emph
e:}\end{picture}\] but it takes three transpositions to transform
the tree $c$ into $e$:
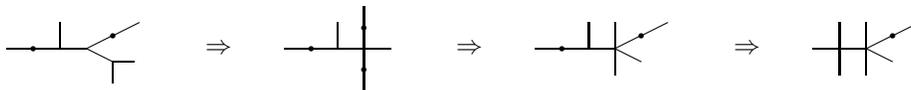
\[\begin{picture}(345,40) \put(0,20){\line(1,0){30}}
\put(10,20){\circle*{2}} \put(20,20){\line(0,1){10}}
\put(30,20){\line(2,1){20}} \put(40,25){\circle*{2}}
\put(30,20){\line(2,-1){10}} \put(40,15){\line(1,0){8}}
\put(40,15){\line(0,-1){8}}

\put(75,18){$\Rightarrow$}

\put(105,20){\line(1,0){40}} \put(115,20){\circle*{2}}
\put(125,20){\line(0,1){10}} \put(135,4){\line(0,1){32}}
\put(135,12){\circle*{2}} \put(135,28){\circle*{2}}

\put(170,18){$\Rightarrow$}

\put(200,20){\line(1,0){30}} \put(210,20){\circle*{2}}
\put(220,20){\line(0,1){10}} \put(230,10){\line(0,1){20}}
\put(230,20){\line(2,1){20}} \put(240,25){\circle*{2}}
\put(230,20){\line(2,-1){10}}

\put(275,18){$\Rightarrow$}

\put(305,20){\line(1,0){20}} \put(315,10){\line(0,1){20}}
\put(325,20){\line(2,1){20}} \put(335,25){\circle*{2}}
\put(325,20){\line(2,-1){10}} \put(325,10){\line(0,1){20}}
\end{picture}\] The tree $a$ can be transformed by one
transposition into the tree $f$
\[\begin{picture}(55,30) \put(15,15){\line(1,0){40}}
\put(25,5){\line(0,1){20}} \put(35,15){\line(0,1){10}}
\put(45,5){\line(0,1){20}} \put(0,12){\small\emph
f:}\end{picture}\] It takes three transpositions to transform the
tree $c$ into the tree $f$:
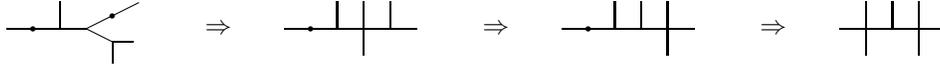
\[\begin{picture}(355,40) \put(0,20){\line(1,0){30}}
\put(10,20){\circle*{2}} \put(20,20){\line(0,1){10}}
\put(30,20){\line(2,1){20}} \put(40,25){\circle*{2}}
\put(30,20){\line(2,-1){10}} \put(40,15){\line(1,0){8}}
\put(40,15){\line(0,-1){8}}

\put(75,18){$\Rightarrow$}

\put(105,20){\line(1,0){50}} \put(115,20){\circle*{2}}
\put(125,20){\line(0,1){10}} \put(135,10){\line(0,1){20}}
\put(145,20){\line(0,1){10}}

\put(180,18){$\Rightarrow$}

\put(210,20){\line(1,0){50}} \put(220,20){\circle*{2}}
\put(230,20){\line(0,1){10}} \put(240,20){\line(0,1){10}}
\put(250,10){\line(0,1){20}}

\put(285,18){$\Rightarrow$}

\put(315,20){\line(1,0){40}} \put(325,10){\line(0,1){20}}
\put(335,20){\line(0,1){10}} \put(345,10){\line(0,1){20}}
\end{picture}\] And it takes nine transpositions to transform the
tree $b$ into the tree $f$:
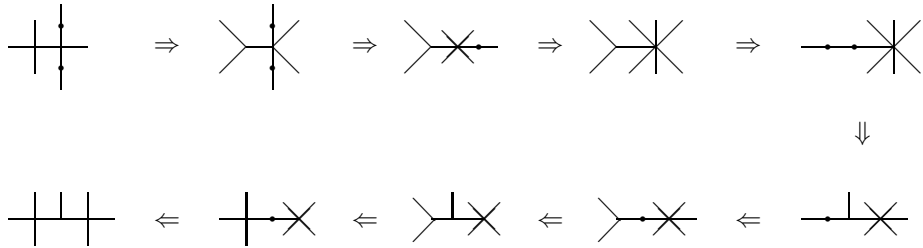
\[\begin{picture}(350,100) \put(0,15){\line(1,0){40}}
\put(10,5){\line(0,1){20}} \put(20,15){\line(0,1){10}}
\put(30,5){\line(0,1){20}}

\put(55,12){$\Leftarrow$}

\put(80,15){\line(1,0){30}} \put(90,5){\line(0,1){20}}
\put(100,15){\circle*{2}} \put(104,9){\line(1,1){12}}
\put(104,21){\line(1,-1){12}}

\put(130,12){$\Leftarrow$}

\put(160,15){\line(1,0){20}} \put(153,16){$\diagdown$}
\put(153,8){$\diagup$} \put(168,15){\line(0,1){10}}
\put(174,9){\line(1,1){12}} \put(174,21){\line(1,-1){12}}

\put(200,12){$\Leftarrow$}

\put(230,15){\line(1,0){30}} \put(223,16){$\diagdown$}
\put(223,8){$\diagup$} \put(240,15){\circle*{2}}
\put(244,9){\line(1,1){12}} \put(244,21){\line(1,-1){12}}

\put(275,12){$\Leftarrow$}

\put(300,15){\line(1,0){40}} \put(310,15){\circle*{2}}
\put(318,15){\line(0,1){10}} \put(324,9){\line(1,1){12}}
\put(324,21){\line(1,-1){12}}

\put(0,80){\line(1,0){30}} \put(10,70){\line(0,1){20}}
\put(20,64){\line(0,1){32}} \put(20,72){\circle*{2}}
\put(20,88){\circle*{2}}

\put(55,78){$\Rightarrow$}

\put(80,70){\line(1,1){10}} \put(80,90){\line(1,-1){10}}
\put(90,80){\line(1,0){10}} \put(100,64){\line(0,1){32}}
\put(100,80){\line(1,1){10}} \put(100,80){\line(1,-1){10}}
\put(100,72){\circle*{2}} \put(100,88){\circle*{2}}

\put(130,78){$\Rightarrow$}

\put(150,70){\line(1,1){10}} \put(150,90){\line(1,-1){10}}
\put(160,80){\line(1,0){25}} \put(164,74){\line(1,1){12}}
\put(164,86){\line(1,-1){12}} \put(178,80){\circle*{2}}

\put(200,78){$\Rightarrow$}

\put(220,70){\line(1,1){10}} \put(220,90){\line(1,-1){10}}
\put(230,80){\line(1,0){15}} \put(245,70){\line(0,1){20}}
\put(235,70){\line(1,1){20}} \put(235,90){\line(1,-1){20}}

\put(275,78){$\Rightarrow$}

\put(300,80){\line(1,0){35}} \put(310,80){\circle*{2}}
\put(320,80){\circle*{2}} \put(335,70){\line(0,1){20}}
\put(325,70){\line(1,1){20}} \put(325,90){\line(1,-1){20}}

\put(320,45){$\Downarrow$}
\end{picture}\] So the subgraph of $G_{10}$ which has all above
trees as vertices and whose edges correspond to above mentioned
transpositions is homeomorphic to $K_{3,3}$.

\vspace{5mm}
\end{document}